\documentclass[letterpaper, 10 pt, conference]{ieeeconf}

\IEEEoverridecommandlockouts
\usepackage{cite}
\usepackage{amsmath,amsthm,amssymb,amsfonts}
\allowdisplaybreaks
\usepackage{algorithmic}
\usepackage{graphicx}
\usepackage{textcomp}
\usepackage{xcolor}
\usepackage{bbm}
\usepackage{balance}
\usepackage{bm}
\usepackage{algorithmic}
\usepackage{algorithm}
\usepackage[caption=false, font=footnotesize]{subfig}
\newtheorem{theorem}{Theorem}

\newtheorem{lemma}{Lemma}
\newtheorem{assumption}{Assumption}
\newtheorem{remark}{Remark}
\newtheorem{corollary}{Corollary}

\newcommand{\blue}[1]{\textcolor{black}{#1}}

\def\BibTeX{{\rm B\kern-.05em{\sc i\kern-.025em b}\kern-.08em
    T\kern-.1667em\lower.7ex\hbox{E}\kern-.125emX}}
\begin{document}


\title{\LARGE \bf Online Feedback Optimization over Networks: \\ A Distributed Model-free Approach\\
\thanks{
\textsuperscript{*}Automatic Control Laboratory, EPFL, Switzerland. Email: wenbin.wang@epfl.ch. \textsuperscript{$\dagger$}Automatic Control Laboratory, ETH Zurich, Switzerland. Email: \{zhiyhe, gbelgioioso, bsaverio, dorfler\}@ethz.ch. This work was supported by the SNSF via NCCR Automation (grant number 180545). Z.~He also acknowledges the support of the Max Planck ETH Center for Learning Systems.}
}
\newcommand{\obj}{\Tilde{\Phi}}
\newcommand{\gest}{\Tilde{\nabla}}
\newcommand{\tr}[1]{\text{tr}(#1)}
\newcommand{\E}[1]{\bb{E}[#1]}
\newcommand{\Ef}[2]{\bb{E}[#1|\mathcal{F}_{#2}]}
\newcommand{\N}{\frac{1}{N}\mathbf{1}\mathbf{1}^{\top}}
\newcommand{\inff}[1]{#1_{\infty}}
\newcommand{\bb}[1]{\mathbb{#1}}
\newcommand{\norm}[1]{\|#1\|}

\author{Wenbin Wang\textsuperscript{*}, Zhiyu He\textsuperscript{$\dagger$}, Giuseppe Belgioioso\textsuperscript{$\dagger$}, Saverio Bolognani\textsuperscript{$\dagger$}, and Florian D{\"o}rfler\textsuperscript{$\dagger$}
}

\maketitle

\begin{abstract}
Online feedback optimization (OFO) enables optimal steady-state operations of a physical system by employing an iterative optimization algorithm as a dynamic feedback controller. When the plant consists of several interconnected sub-systems, centralized implementations become impractical due to the heavy computational burden and the need to pre-compute system-wide sensitivities, which may not be easily accessible in practice. Motivated by these challenges, we develop a fully distributed model-free OFO controller, featuring consensus-based tracking of the global objective value and local iterative (projected) updates that use stochastic gradient estimates. We characterize how the closed-loop performance depends on the size of the network, the number of iterations, and the level of accuracy of consensus. Numerical simulations on a voltage control problem in a direct current power grid corroborate the theoretical findings.


\end{abstract}


\section{Introduction}
Optimal steady-state operations of dynamical systems are a frequent and important engineering task, with examples from power grids, traffic networks, to mobile robot swarms. These systems are inherently complex, preventing us from constructing accurate models of their behaviors. Moreover, exogenous disturbances that act on systems are unknown in general, which adds to the difficulty of formulating explicit input-output maps. Therefore, computing optimal steady-state control inputs offline via numerical solvers and then applying them to the system in a feedforward manner is often suboptimal, lacking real-time adaptability and efficiency.

Online feedback optimization (OFO) \cite{Hauswirth2021,simonetto2020time,krishnamoorthy2022real} is an emerging control paradigm to steer complex dynamical systems to optimal steady-state operating points. Its core principle is to interconnect optimization algorithms in closed loop with a physical plant, thus circumventing the need for an explicit model of the plant and the exact value of the disturbance. OFO enjoys closed-loop optimality and stability in various scenarios, e.g., convex \cite{simpsonporco2020, hauswirth2020antiwindup,Hauswirthadrian2021} and non-convex problems \cite{Haberleverena2021,Belgioioso2022,10354356} with input\cite{bianchin2021time} and output constraints\cite{bernstein2019online}, as well as linear \cite{lawrence2020linearconvex, Colombino2020} and non-linear dynamical systems\cite{Hauswirthadrian2021}. The effectiveness of OFO has been empirically demonstrated in real-world power distribution grids\cite{ortmann2023deployment}.


The above OFO controllers are mostly centralized. When applied to large-scale networked systems involving interconnected subsystems, they can face issues of scalability, robustness, and privacy disclosure, which are common in large-scale centralized decision architectures \cite{Nedicangelia2018}. First, a central processor is needed to communicate with all the subsystems and generate network-wide control inputs. Thus, there exist severe computational and communication bottlenecks for sizable networks. Second, this centralized architecture is not robust against potential failures or attacks at the central unit. Finally, since gradient-based OFO controllers (e.g., \cite{bianchin2021time,bernstein2019online,Hauswirth2021,Haberleverena2021}) require the system-wide input-output sensitivity to perform local iterative updates, concerns on privacy disclosure and local modeling also arise.


To tackle the issues of the centralized paradigm, multiple decentralized or distributed OFO controllers have been developed. Some decentralized methods\cite{wang2023decentralized,Belgioioso2022,belgioioso2023tutorial} leverage local decision-making without inter-agent communication. While these methods ensure convergence to competitive equilibria (e.g., Nash), they experience sub-optimality compared to the globally optimal solutions. Other distributed approaches exploit consensus-based information exchange for coordination among neighbors \cite{chang2019saddle,bernstein2019real,zhang2018distributed}. The aforementioned controllers all use first-order iterations, and, therefore, still require local input-output sensitivities of agents. In terms of large-scale complex systems, the estimation of those local sensitivities can be highly non-trivial, prone to errors, or prohibitive due to lack of measurements, nonlinearity, and non-smoothness of input-output maps. To address unknown sensitivities, another line of works \cite{10354356,poveda2017robust,chen2020model,chen2021safe} explore model-free OFO via stochastic gradient estimates developed in zeroth-order optimization \cite{Nesterov2017,Zhang2020,chen2022improve}. Nonetheless, how to extend these centralized model-free methods to achieve distributed decision-making remains unclear.

In this paper, we propose a distributed model-free OFO controller to optimize the steady-state performance of a networked nonlinear system. To circumvent local sensitivities and satisfy input constraints, we leverage local iterative updates based on gradient estimates and projection. More importantly, to achieve fully distributed operations, we design a new protocol for maintaining and updating a queue of historical local objective values, which facilitates each agent to locally track the global objective value and construct gradient estimates. Our work is aligned with \cite{Tang2020} that pursues cooperative optimization over networks. In contrast, we consider the input-output coupling between agents through dynamics and attain lower costs of local communication and storage, because we do not require a table that records the information of all the agents as \cite{Tang2020}. We characterize the optimality of the closed-loop interconnection between the proposed controller and the networked system. Specifically, we quantify the distance to the optimal point as a function of the size of the network, the number of iterations, and the accuracy of consensus. We further apply the proposed distributed controller to a voltage control problem for a direct current power system.

\section{Problem Formulation and Preliminaries}
Consider a networked system consisting of $N$ subsystems, or agents. The system is described by a graph $\mathcal{G}(\mathcal{N},\mathcal{E})$, where $\mathcal{N}$ and $\mathcal{E} \subseteq \mathcal{N} \times \mathcal{N}$ denote the set of agents and the set of edges, respectively. Agent $i$ can communicate with agent $j$ if and only if $(i,j) \in \mathcal{E}$, where $i,j \in \mathcal{N}$. The neighborhood of agent $i$ is represented by $\mathcal{N}_i$.
Given a vector $w \in \mathbb{R}^N$, $w_k$ denotes its value at the $k$-th iteration, while $w(i)$ represents its $i$-th element.

\subsection{Problem Formulation}
\blue{Let the steady-state input-output map of a networked system be given by a nonlinear function $h:\bb{R}^N \times \bb{R}^N \rightarrow \bb{R}^N$}
\begin{equation}\label{eq:steady-state-map}
    y=h(u,d),
\end{equation}
\blue{where $u \in \bb{R}^N$ is the input, $y \in \bb{R}^N$ is the output, $d\in \bb{R}^N$ is a constant unknown exogenous disturbance drawn from a certain distribution. The $i$-th elements of $u$, $y$, and $d$ are the local input, output, and disturbance of agent $i$, respectively, where $i=1,\ldots,N$.}
We focus on a per-agent scalar setting to simplify notation, however, our results readily extend to scenarios involving multiple variables.

We aim to optimize the steady-state input-output performance of the networked system \eqref{eq:steady-state-map}, i.e., 
\begin{subequations}
    \label{prob: introduction 2}
    \begin{align}
    \min_{u\in\bb{R}^N,y\in\bb{R}^N} \quad & \Phi(u,y) \triangleq \frac{1}{N}\sum_{i=1}^{N}\Phi_i(u(i),y(i)) \\
    \textrm{s.t.} \quad & y=h(u,d), \label{eq:ss-map} \\
    & u \in \mathcal{U} \triangleq {\textstyle \prod_{i=1}^{N}} \mathcal{U}_i.
    \end{align}
\end{subequations}
In problem~\eqref{prob: introduction 2}, the global objective function $\Phi: \bb{R}^N \times \bb{R}^N \rightarrow \bb{R}$ is the average of all the local objectives $\Phi_i: \bb{R} \times \bb{R} \rightarrow \bb{R}$, and $u(i)$ and $y(i)$ are the input and output of agent $i$, respectively. Moreover, $\mathcal{U}_i \subseteq \mathbb{R}$ is the local input constraint set of agent $i$. The overall constraint set $\mathcal{U}$ is the Cartesian product of all the local sets.
By incorporating \eqref{eq:ss-map} into the objective $\Phi(u,y)$, problem \eqref{prob: introduction 2} can be transformed to:
\begin{equation}
    \label{prob: introduction 3}
    \begin{split}
        \min_{u\in\bb{R}^N} \quad& \obj(u) \\
        \text{s.t.} \quad& u \in \mathcal{U},
    \end{split}
\end{equation}
where $\obj(u)=\frac{1}{N}\sum_{i=1}^{N}\obj_i(u)$, and $\tilde \Phi_i(u) \triangleq \Phi_i(u(i),y(i))$ is the reduced local objective. \blue{Also, $y(i)$ denotes the $i$-th element of the overall output $y=h(u,d)$.} The disturbance $d$ acts as an unknown parameter of $\obj(u)$. We make the following assumptions on the objectives and the network. 

\begin{assumption}
\label{assumption one}
    The reduced objective $\obj_i$ of agent $i$ is $L_0$-Lipschitz, $L_1$-smooth, and $m$-strongly convex.
\end{assumption}

\begin{assumption}
\label{assumption two}
    The graph $\mathcal{G}$ describing the networked system \eqref{eq:steady-state-map} is undirected and connected.
\end{assumption}

Both Assumptions \ref{assumption one} and \ref{assumption two} are standard in the literature of network control, see \cite{Maritan2023,shamir2015optimal, Tang2020} and \cite{Nedicangelia2018,Tang2019}, respectively. \blue{For instance, the assumption that $\obj_i$ is strongly convex can be satisfied when the system has linear dynamics and the objective $\Phi_i$ is a strongly convex function of $u(i)$ and $y(i)$.}

Solving problem \eqref{prob: introduction 2} with numerical optimization algorithms requires accurate knowledge of the static model \(h\) and the disturbance \(d\), which can be impractical in real-world applications
 \cite{Hauswirth2021}. In contrast, in this paper we aim to develop a fully-distributed OFO controller that leverages real-time output measurements $y$ to bypass the need for any model information related to the input-output map $h$.

\subsection{Gradient Estimation}
The key technique that enables model-free operations is zeroth-order optimization, where a gradient estimate, constructed based on objective values and random exploration vectors, is used in the iterative update. In this paper, we utilize the one-point residual feedback estimate proposed in \cite{Zhang2020}. This scheme proposes to estimate the gradient of a generic cost function $\zeta:\bb{R}^N\rightarrow\bb{R}$, evaluated at $u_k$, using
\begin{equation}
\label{eq: one-point residual estimation}
    \gest \zeta(u_{k}) = \frac{\zeta(u_k + \delta v_k) - \zeta(u_{k-1} + \delta v_{k-1})}{\delta} v_k,
\end{equation}
where $u_k,u_{k-1}\in \mathbb{R}^N$ are optimization variables, $\gest \zeta(u_k)$ is the gradient estimate, $\delta>0$ is the so-called smoothing parameter, and $v_k,v_{k-1}$ are independent and identically distributed random vectors sampled from a multivariate normal distribution. The expectation of \eqref{eq: one-point residual estimation} equals the gradient of the smooth approximation $\zeta_{\delta}$ of $\zeta$, where $\zeta_{\delta}(u):=\bb{E}_{v\sim\mathcal{N}(0,I)}[\zeta(u+\delta v)]$, and $I \in \mathbb{R}^{N \times N}$ is the identity matrix. If $\zeta$ is convex and $L$ smooth, then $\zeta_{\delta}$ is also convex and $L$ smooth. Moreover, the smooth approximation $\zeta_{\delta}$ can be bounded as $\zeta(u) \leq \zeta_{\delta}(u) \leq \zeta(u) + \frac{L}{2}\delta^2 N, \forall u \in \mathbb{R}^N$.

\section{Design of Distributed Model-free OFO Controllers}
The centralized model-free OFO \cite{10354356,chen2020model,chen2021safe} for solving \eqref{prob: introduction 3} requires a central processor to collect local objective values $\obj_i(u)$, evaluate the global objective to construct a gradient estimate via \eqref{eq: one-point residual estimation}, and then broadcast the updated solution. In practice, however, this centralized design faces challenges related to scalability, robustness, and privacy preservation\cite{Nedicangelia2018}.

In contrast, we pursue an OFO controller to regulate the system \eqref{eq:steady-state-map} in a distributed manner. This controller employs a novel communication protocol, where each agent stores and exchanges a vector of historical evaluations of local objectives for tracking the global objective value. This vector is updated via average consensus in the same time scale as the local optimization iterations. In the remainder of this paper, we treat the unconstrained case of \eqref{prob: introduction 3} (i.e., $\mathcal{U} = \mathbb{R}^N$) separately, to derive sharper results.



\subsection{Unconstrained Setting}

We first consider problem \eqref{prob: introduction 3} in the unconstrained scenario, i.e., $\mathcal{U} = \mathbb{R}^N$. The key idea of the proposed distributed OFO controller is to exchange and update local estimates of the global objective values, construct gradient estimates similar to \eqref{eq: one-point residual estimation}, and then iteratively adjust local control inputs.

\begin{algorithm}[!t]
    \renewcommand{\algorithmicrequire}{\textbf{Input:}}
    \renewcommand{\algorithmicensure}{\textbf{Output:}}
    \caption{Model-free distributed OFO controller}
    \label{algortihm: model-free controller}
    \begin{algorithmic}[1]
        \REQUIRE consensus parameter $\tau$, step size $\eta$, smoothing parameter $\delta$, $u_0 \in \bb{R}^N$, initial vector $z_{i,0}\in \bb{R}^{\tau}$ of each agent, total number of iterations $T$
        \FOR{$k=0, \ldots, T-1$}
        \FOR {every agent $i$}
            \STATE evaluate $\obj_i(u_k+\delta v_k) = \Phi_i(u_k(i)+\delta v_k(i),y_k(i))$;
            \STATE $\bm{z}_{i,k}(l) = \sum_{j \in \mathcal{N}_i}W_{ij}\bm{z}_{j,k}(l)$, $\forall l =1,\ldots,\tau$; \label{algo: consensus averaging}
            \STATE $\bm{z}^{(1)}_{i,k} = [\bm{z}^{\top}_{i,k},\obj_i(u_k+\delta v_k)]^{\top}$;\label{algo: adding new element}
            \STATE $u_{k+1}(i) = u_k(i) - \eta\frac{\Delta_k(i)}{\delta}$, where
            \begin{equation}\label{eq:Delta_cases}
            \Delta_k(i) \!=\!
            \begin{cases}
                \bm{z}^{(1)}_{i,k}(1)v^{\text{init}}_0(i)&\text{if $k\!=\! 0$},\\
                \big(\bm{z}^{(1)}_{i,k}(1)\!-\!\bm{z}^{(1)}_{i,k-1}(1)\big)v^{\text{init}}_{k}(i)&\text{if $0\!<\!k\!<\!\tau$},\\
                \big(\bm{z}^{(1)}_{i,k}(1)\!-\!\bm{z}^{(1)}_{i,k-1}(1)\big)v_{k-\tau}(i)&\text{if $k\!\geq\!\tau$};
            \end{cases}
            \end{equation}\label{algo: line 7}
            \STATE $\bm{z}_{i,k+1} = [\bm{z}^{(1)}_{i,k}(2),\hdots,\bm{z}^{(1)}_{i,k}(\tau+1)]^{\top}$;\label{algo: line 14}
        \ENDFOR
        \ENDFOR
        \ENSURE $u_T$
    \end{algorithmic}
\end{algorithm}

Let $\bm{z}_{i,k}$ be a vector containing the local objective values of agent $i$ at iteration $k$ and $\bm{z}_{i,k}(j)$ be the $j$-th entry of $\bm{z}_{i,k}$. The model-free distributed controller requires a suitable initialization of $\bm{z}_{i,0}\in \bb{R}^{\tau}$, where $\tau\in \mathbb{N}$ is a design parameter. For this, agent $i$ can evaluate the local objective $\tau$ times using $v^{\text{init}}_l(i) \sim \mathcal{N}(0,1)$ for $l =0,\hdots,\tau-1$ and stack the result in $\bm{z}_{i,0}$. Consequently, we have 
$$\bm{z}_{i,0} = \Big[\obj_i(u_0+\delta v^{\text{init}}_0),\hdots,\obj_i(u_0+\delta v^{\text{init}}_{\tau-1})\Big]^{\top}.$$

We summarize the detailed steps of the proposed distributed OFO controller in Algorithm \ref{algortihm: model-free controller}. \blue{The shared coefficients, such as $\tau, \eta$, and $\delta$, can be either set in advance or synchronized via consensus mechanisms.}
At every iteration $k$, each agent $i$ generates a random exploration signal $v_k(i) \sim \mathcal{N}(0,1)$ synchronously and evaluates its local objective value $\obj_i(u_k+\delta v_k) = \Phi_i(u_k(i)+\delta v_k(i),y_k(i))$ with its local input $u_k(i) + \delta v_k(i)$ and local measured output $y_k(i)$, which captures the collective effect of the coupled dynamics. These local objective values are stored in a vector $\bm{z}_{i,k}\in \bb{R}^{\tau}$ of size $\tau$. Then, the agents update $\bm{z}_{i,k}$ via the average consensus step (see line \ref{algo: consensus averaging} in Algorithm \ref{algortihm: model-free controller}), where $W$ is a doubly stochastic matrix with Metropolis weights\cite{Nedicangelia2018}. 
Afterward, the new local evaluation of $\obj_i$ queues into $\bm{z}_{i,k}$. We denote this vector with $\tau + 1$ elements by $\bm{z}^{(1)}_{i,k}$. The first element of $\bm{z}^{(1)}_{i,k}$ is taken out to facilitate the optimization iterations (see lines \ref{algo: line 7}-\ref{algo: line 14} in Algorithm \ref{algortihm: model-free controller}). We illustrate the process of communicating and updating $\bm{z}_{i,k}$ in Fig.~\ref{fig: distribute algo vector}.

In fact, when $k\geq \tau+1$, the first element of $\bm{z}^{(1)}_{i,k}$ corresponds to performing $\tau$ consensus iterations for the local objective values at time $k\!-\!\tau$. This element closely approximates the global objective value at time $k\!-\!\tau$ with a suitable $\tau$, since $\bm{z}^{(1)}_{i,k}(1)=\sum_{j=1}^NW^{\tau}_{ij} \obj_j(u_{k-\tau}+\delta v_{k-\tau})$. 

We write $\bm{Z}_k = [\bm{z}^{(1)}_{1,k}(1),\dots,\bm{z}^{(1)}_{N,k}(1)]^{\top}$, where $\bm{Z}_k\in \bb{R}^{N}$. For $k \geq \tau + 1$, the update for all the agents is
\begin{equation}
\label{algo: delayed zeroth-order}
u_{k+1} = u_{k} - \eta \frac{1}{\delta} \big(\underbrace{\bm{Z}_k}_{W^\tau\varphi_{k\!-\!\tau}}-\underbrace{\bm{Z}_{k-1}}_{W^\tau\varphi_{k\!-\!\tau\!-\!1}}\big)\odot v_{k-\tau},
\end{equation}
where $\varphi_{k-\tau} = [\obj_1(u_{k-\tau}+\delta v_{k-\tau}),\hdots,\obj_N(u_{k-\tau}+\delta v_{k-\tau})]^{\top}$ is of size $N$, and $x\odot y$ represents the Hadamard product between two vectors $x$ and $y$.

\begin{remark}
     The parameter $\tau$ regulates the trade-off between the accuracy of the gradient estimate and the communication burden. On the one hand, if $\tau$ is large, then $W^{\tau}$ approximates $\N$ well, where $\mathbf{1} \in \mathbb{R}^N$ is an all-ones vector. Thus, the distributed update \eqref{algo: delayed zeroth-order} owns a comparable accuracy guarantee as that of its centralized counterpart. On the other hand, a large $\tau$ indicates an increased amount of data exchanged between neighboring agents at each iteration. In practice, we can choose $\tau$ according to the bandwidth of the communication channel and the accuracy requirement.
\end{remark}

\begin{figure}[!t]
    \centering
    \includegraphics[width=0.9\linewidth]{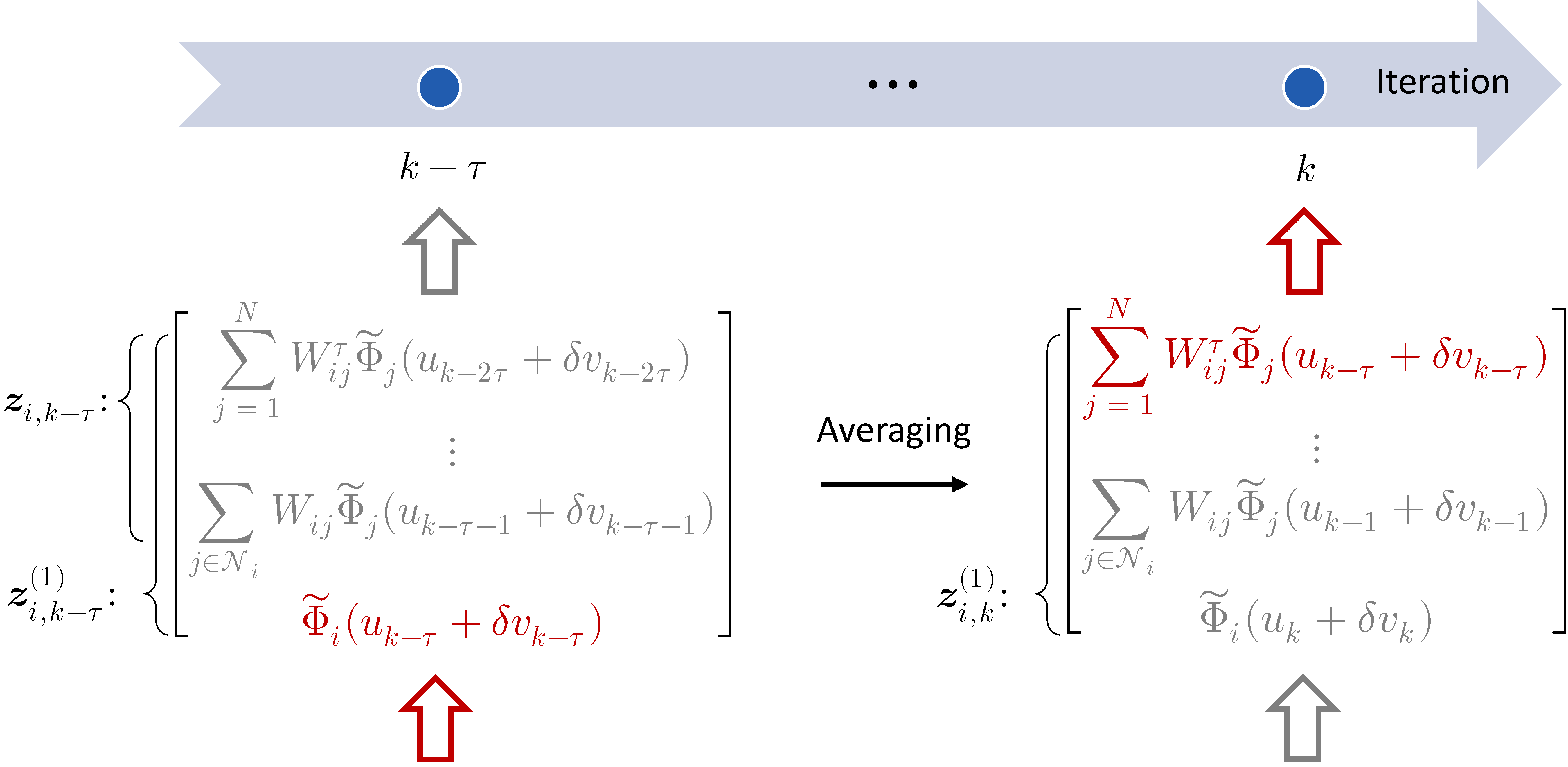}
    \caption{Illustration of the communication protocol.}
    \label{fig: distribute algo vector}
\end{figure}

\subsection{Constrained Setting}
In many applications, there exist constraints on the admissible input. Such constraints can arise, for instance, from the upper and lower bounds imposed by actuation limits. We make the following assumption on the local input constraint set $\mathcal{U}_i$, where $i=1,\ldots,N$. 
\begin{assumption}
\label{assumption three}
    The set $\mathcal{U}_i \subseteq \bb{R}$ is convex, closed, and bounded with a diameter $D_i$, i.e., $\forall u,u' \in \mathcal{U}_i, |u-u'| \leq D_i$.
\end{assumption}
Assumption~\ref{assumption three} is commonly used in the literature, e.g., \cite{Tang2020,Haberleverena2021,Dixit2018}, and satisfied in applications. It implies the overall constraint set $\mathcal{U} = \prod_{i=1}^{N} \mathcal{U}_i$ is bounded by  $D_\mathcal{U} \triangleq \sqrt{\sum_{i=1}^{N} D_i^2}$.

To handle the input constraint, we extend the proposed distributed controller \eqref{algo: delayed zeroth-order} with projection. Namely, every agent $i$ performs the following update:
\begin{equation}
    \label{algo: projected zeroth-order}
    u_{k+1}(i) = \text{Proj}_{\mathcal{U}_i}\Big[u_k(i) - \frac{\eta}{\delta}\Delta_k(i)\Big],
\end{equation}
where $\text{Proj}_{\mathcal{U}_i}[\cdot]$ denotes the projection onto $\mathcal{U}_i$, and $\Delta_k$ is defined as in \eqref{eq:Delta_cases}.


\section{Performance Guarantees}
We provide convergence guarantees when the proposed distributed model-free controller \eqref{algo: delayed zeroth-order} is implemented in the closed loop with system \eqref{eq:steady-state-map}. We use the expected distance to the unique solution $u^*$ to problem \eqref{prob: introduction 3}, i.e., $\bb{E}[\norm{u_k-u^*}^2]$, as the convergence measure. This metric is common in the literature when the objective is strongly convex and smooth \cite{Tang2020,Nedicangelia2018}.

\label{section: stability analysis zeroth order}
\subsection{Unconstrained Setting}
We first focus on the unconstrained scenario of problem \eqref{prob: introduction 3}, with $\mathcal{U} = \mathbb{R}^N$.
Let the centralized gradient estimate as per \eqref{eq: one-point residual estimation} and the consensus error be defined by
\begin{flalign*}
    \qquad&\gest \obj(u_{k-\tau})=\frac{1}{\delta}\N
    (\varphi_{k-\tau}-\varphi_{k-\tau-1})\odot v_{k-\tau},\\
    &e_{k-\tau} = \frac{1}{\delta}(W^{\tau}-\N)(\varphi_{k-\tau}-\varphi_{k-\tau-1})\odot v_{k-\tau}.&
\end{flalign*}
Therefore, the overall update \eqref{algo: delayed zeroth-order} can be transformed to
\begin{equation}
\label{eq: compact form}
    u_{k+1} = u_{k}- \eta(\tilde{\nabla}\tilde{\Phi}(u_{k-\tau}) + e_{k-\tau}),
\end{equation}
where $e_{k-\tau}$ represents the error relative to the centralized gradient estimate arising from a finite number of consensus iterations. We further provide upper bounds on the second moments of gradient estimates and errors.

\begin{lemma}
    \label{lemma 1}
    Let Assumptions \ref{assumption one} and \ref{assumption two} hold and $\mathcal{U} = \mathbb{R}^N$. If $0<\eta<\frac{\delta}{\sqrt{4NL_0^2\tr{W^{2\tau}}}}$, then for any $k>1$, $\E{\norm{\gest\obj(u_k)+e_k}^2}$, $\E{\norm{\gest \obj(u_k)}^2}$, and $\E{\norm{e_k}^2}$ are bounded from above by $R$, $R_f$, and $R_e$, respectively, where $R = \mathcal{O}(N^2)$, $R_f = \mathcal{O}(N^2)$, and $R_e = \mathcal{O}(N^2,\text{tr}[(W^{\tau}-\N)^{2}])$ are constants, which can be found in Appendix \ref{appendix: proof of technical results}.
\end{lemma}
\begin{proof}
    The proof can be found in Appendix \ref{appendix: proof of technical results}.
\end{proof}

Note that $R_e$ exhibits a polynomial dependence on the number of agents $N$. Furthermore, the term $\text{tr}[(W^{\tau}-\N)^{2}]$ describes how the consensus error decreases when $\tau$ increases. For $\eta$ and $\delta$ satisfying the conditions of Lemma~\ref{lemma 1}, we can choose $\tau$ such that the error related to $e_k$ is arbitrarily small.

Next, we characterize the optimality of the closed-loop interconnection of the proposed distributed controller \eqref{algo: delayed zeroth-order} and the system \eqref{eq:steady-state-map}.

\begin{theorem}
\label{theorem: delayed inexact gradient estimate}
    Given Assumptions \ref{assumption one} and \ref{assumption two}, if $\mathcal{U} = \mathbb{R}^N$, $m>1$, and $0<\eta<\frac{\delta}{\sqrt{4NL_0^2\tr{W^{2\tau}}}}$, then for any $k\geq \tau + 1$, the closed-loop system satisfies
    \begin{equation}\label{eq:upper_bound_unconstrained}
        \E{\|u_{k+1}-u^*\|^2} \leq \rho^{k-\tau}\E{\|u_{\tau+1}-u^*\|^2} + \frac{p(\eta)}{1-\rho},
    \end{equation}
    where $\rho \!=\! 1-(m-1)\eta + m\eta^2$, $p(\eta) = a_1\eta^3+a_2\eta^2+a_3\eta$, and $a_1,a_2,a_3$ are constants given by \eqref{bound on as} in Appendix~\ref{appendix: proof of theorem 1}.
\end{theorem}
\begin{proof}
    The proof can be found in Appendix \ref{appendix: proof of theorem 1}.
\end{proof}

\begin{remark}
    Theorem~\ref{theorem: delayed inexact gradient estimate} requires that the strong convexity coefficient of $\tilde{\Phi}_i$ satisfies $m>1$. For those objectives that do not meet this requirement, we can scale them by some proper constants bigger than $1$, so that this condition is satisfied and the optimal point remains unchanged.
\end{remark}

\blue{The inequality \eqref{eq:upper_bound_unconstrained} suggests that the transient behavior of our distributed algorithm is captured by the convergence of $\E{\|u_{k+1}-u^*\|^2}$ with a linear rate $\rho$. Further, the asymptotic behavior reflecting the suboptimality of the limiting point is dominated by errors due to stochastic gradient estimation and consensus-based tracking mechanisms. We further show that the bound of suboptimality (i.e., $p(\eta)/(1-\rho)$) asymptotically converges to zero as $\tau$ goes to infinity.}

\begin{corollary}
\label{corollary one}
    Consider a sufficiently small $\epsilon > 0$. Let
    \begin{align*}
        \eta_{\epsilon} &= (\sqrt{(a_2\!+\!0.5m\epsilon)^2\!+\!(2a_1\!+\!32N^2L_1L_0^2\tr{W^{2\tau}})(m\!-\!1)\epsilon} \\
        &\quad -(a_2\!+\!0.5m\epsilon))/(2a_1\!+\!32N^2L_1L_0^2\tr{W^{2\tau}}), \\
        \delta &= 2\sqrt{4NL_0^2\tr{W^{2\tau}}}\eta_{\epsilon}.
    \end{align*}
    By choosing $\tau > \ln(\frac{(m-1-m\eta_{\epsilon})\epsilon}{2R_f(N-1)})/\ln(\lambda_2^2)$, we have $\frac{p(\eta_{\epsilon})}{1-\rho}<\epsilon$, where $\lambda_2$ is the second largest eigenvalue of $W$.
\end{corollary}

\begin{proof}
    The proof can be found in Appendix \ref{appendix: proof of corollary}.
\end{proof}


\subsection{Constrained Setting}
Next, we consider the problem with an input constraint set, i.e., $\mathcal{U} \subset \mathbb{R}^N$. Similar to Theorem~\ref{theorem: delayed inexact gradient estimate}, the following theorem provides an upper bound on the expected distance to the optimal point when the projected controller \eqref{algo: projected zeroth-order} is interconnected in closed loop with the system \eqref{eq:steady-state-map}. 

\begin{theorem}
\label{theorem: projected dynamics}
    Let Assumptions \ref{assumption one}, \ref{assumption two}, and \ref{assumption three} hold. If $0< \eta<\frac{\delta}{\sqrt{4NL_0^2\tr{W^{2\tau}}}}$, the closed-loop interconnection attains the following upper bound 
    \begin{equation}\label{eq:dist_opt_constrained}
        \E{\|u_{k+1}-u^*\|}\leq {\rho'}^{k-\tau} \E{\|u_{\tau+1}-u^*\|}+\frac{\eta R'}{1-\rho'},
    \end{equation}
    where $\rho' \!=\! \sqrt{1\!-\!2m\eta\!+\!L_1^2\eta^2}$, and $R'$ is a constant given by
    \begin{equation*}
        R' = 2L_0 \!+\! 4L_0\sqrt{\text{tr}[W^{2\tau}](\frac{ND_{\mathcal{U}}}{\delta^2}\!+\! 4(N\!+\!4)^2)}.
    \end{equation*}
\end{theorem}
\begin{proof}
    The proof can be found in Appendix \ref{appendix: proof of theorem 2}.
\end{proof}

With input constraints in place, the limiting bound on the distance to the optimal point (i.e., $\eta R'/(1-\rho')$) depends on the step size $\eta$ and the diameter of the feasible set $D_{\mathcal{U}}$. \blue{Compared to the unconstrained case, the distance between any two points in the feasible set is not more than $\mathcal{D_{\mathcal{U}}}$ in the constrained case, which eliminates the need for dealing with coupling transients in $u$. While the use of $\mathcal{D_{\mathcal{U}}}$ facilitates analysis, the upper bound \eqref{eq:dist_opt_constrained} tends to be more conservative than the bound \eqref{eq:upper_bound_unconstrained}.}

\section{Application in Distributed Voltage Control}
\label{Numerical Results}
Consider the tree-shaped direct current (DC) power system illustrated in Fig.~\ref{fig: chapter 2 8 node system}. The line is modeled through a resistor and an inductor. At each node, a droop controller (i.e., a capacitor and a resistor connected to the ground) is applied such that the system is stabilized for any current injection\cite{JinxinZhao}. 
The system dynamics is given as
\begin{equation}
\begin{split}
    \label{eq: power system formulation}
    \begin{bmatrix}
        C&0\\
        0&L
    \end{bmatrix}
    \begin{bmatrix}
        \dot{V}\\
        \dot{f}
    \end{bmatrix}=&
    \begin{bmatrix}
        G&-B\\
        B^\top&-R
    \end{bmatrix}
    \begin{bmatrix}
        V\\
        f
    \end{bmatrix}+
    \begin{bmatrix}
        I^* + I_c\\
        0
    \end{bmatrix},\\
    V_{\text{m}}=&V+d,
\end{split}
\end{equation}
where $V \in \bb{R}^8$ is the node voltage, $f\in \bb{R}^7$ is the line current, $I^*\in \bb{R}^8$ is the reference current injection, $I_c\in \bb{R}^8$ is controllable current injection at each node, and $V_{\text{m}} \in \bb{R}^8$ is the voltage measurement corrupted by an unknown constant noise $d\in \bb{R}^8$. We denote the capacitance and resistance of each node by the diagonal matrices $C\in\bb{R}^{8\times8}$ and $G\in\bb{R}^{8\times8}$, respectively. Moreover, the inductance and resistance of each line are represented by the diagonal matrices $L\in \bb{R}^{7\times7}$ and $R\in \bb{R}^{7\times7}$, respectively. The network structure of the power system is denoted by the incidence matrix $B\in \bb{R}^{8\times7}$.

When there is an unknown load change $\Delta I \in \mathbb{R}^8$ in $I^*$, we aim to control the voltage of the system \eqref{eq: power system formulation} close to a reference value $V_{\text{m, ref}}$ while minimizing the effort associated to controllable current injection. The problem is formulated as follows
\begin{equation}
    \label{prob: quadratic function DC power system}
    \begin{aligned}
    \min_{I_c,V_{\text{m}}} \quad & \frac{1}{2}(\norm{I_c}^2+\norm{V_{\text{m}}-V_{\text{m, ref}}}^2)\\
    \textrm{s.t.} \quad & V_{\text{m}}=H(I_c+I^*-\Delta I)+d,\\
    &I_c \in \mathcal{U},
    \end{aligned}
\end{equation}
where the input-output steady-state sensitivity matrix $H =
    \begin{bmatrix}
        I&0
    \end{bmatrix}
    \begin{bmatrix}
        G&-B\\
        B^\top&-R
    \end{bmatrix}^{-1}
    \begin{bmatrix}
        I\\
        0
    \end{bmatrix}\in \bb{R}^{8\times 8}$,
and $V_{\text{m, ref}}=HI^*+d$. Note that local agents do not know the dynamics model \eqref{eq: power system formulation} or the sensitivity matrix $H$. Instead, they utilize local output measurements and collaboratively solve problem \eqref{prob: quadratic function DC power system}. In simulations, the communication network is aligned with the physical structure, and $C$, $G$, and $L$ are set as identity matrices. Further, $I^*$ and $\Delta I$ are all-ones vectors, and the resistance $R$ of each line is $10$. We apply Euler's forward method to discretize the system \eqref{eq: power system formulation} by setting the discretization step size $\epsilon=0.1$. We set the step size $\eta= 0.001$ and smoothing parameter $\delta = 0.002$.




\begin{figure}[!t]
    \centering
    \includegraphics[width=0.8\linewidth]{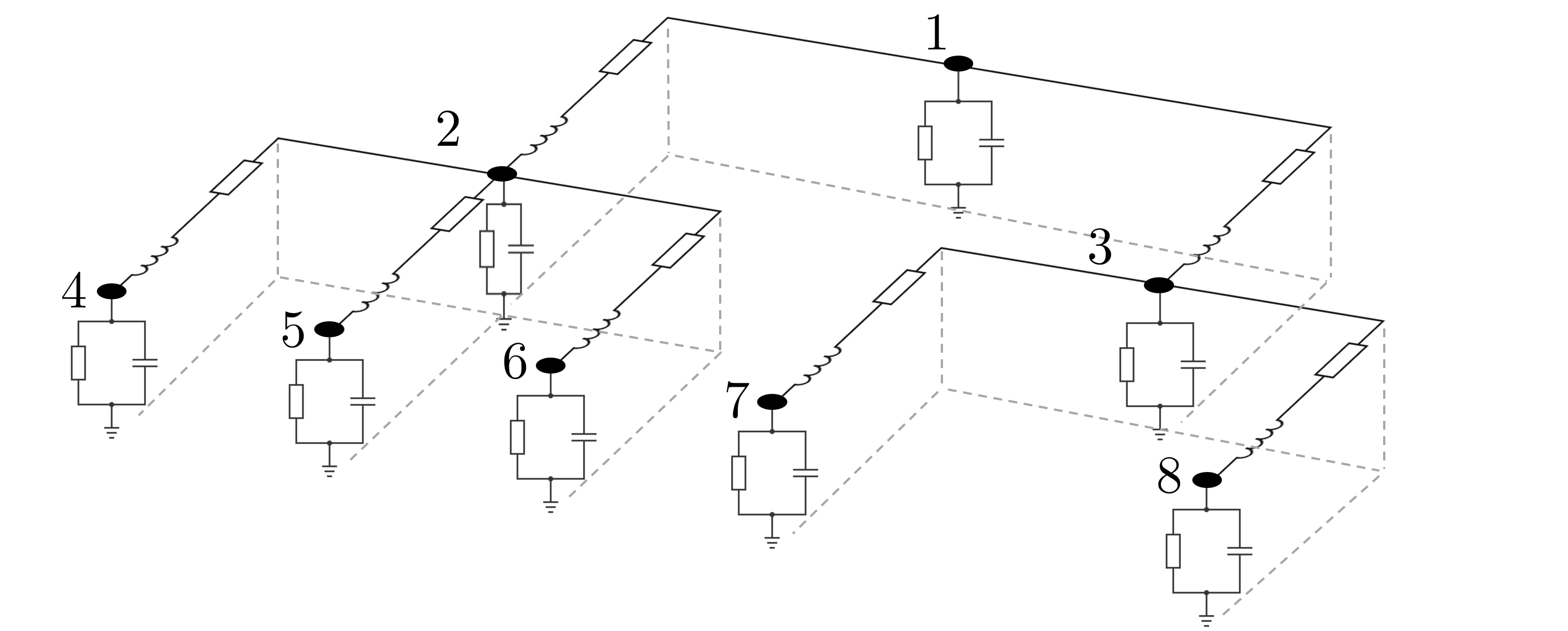}
    \caption{An 8-node DC power system.}
    \label{fig: chapter 2 8 node system}
\end{figure}

\begin{figure} 
    \centering
  \subfloat[Input error (unconstrained)\label{fig: unconstrained}]{%
       \includegraphics[width=0.5\linewidth]{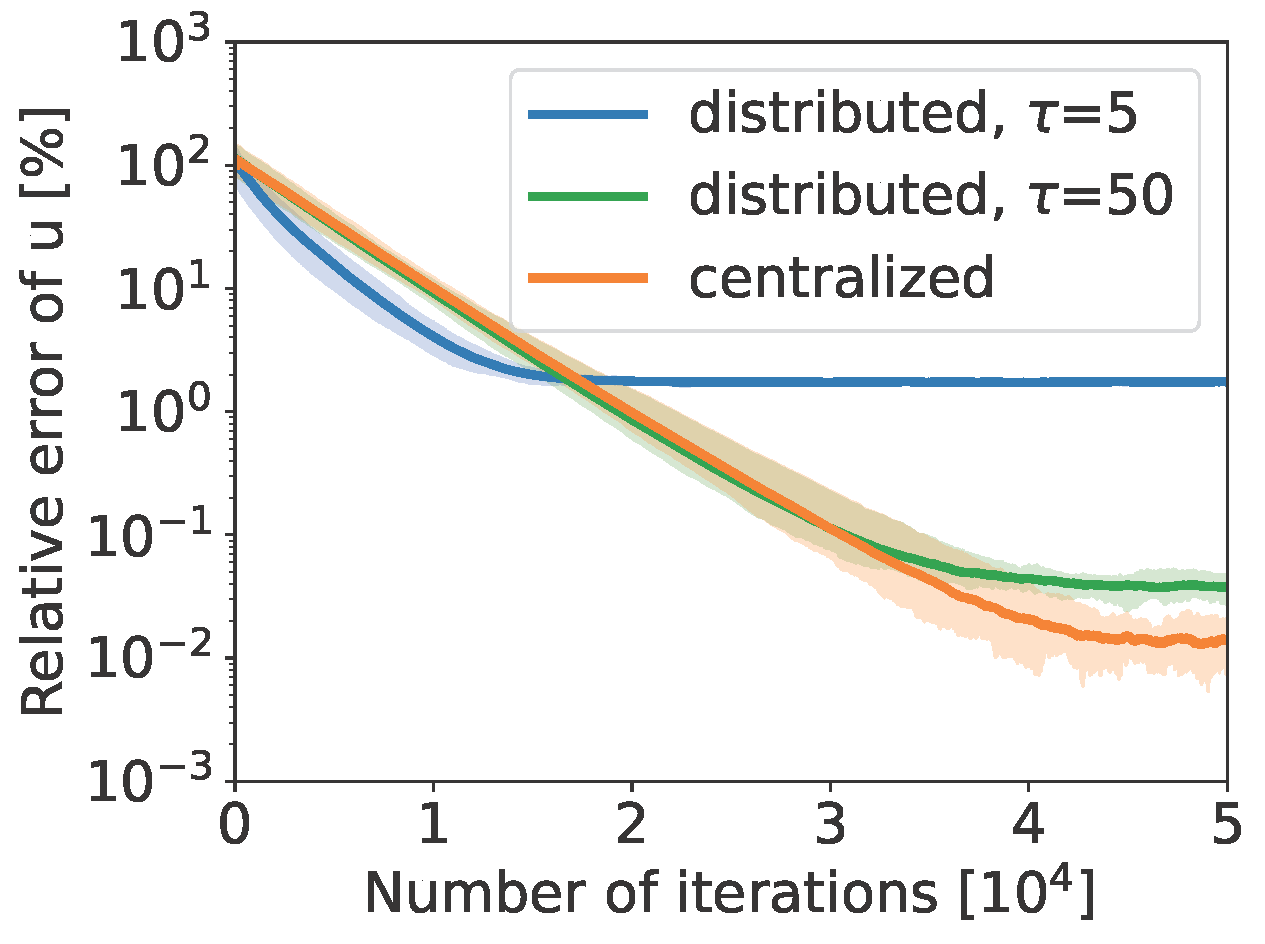}}
    \hfill
  \subfloat[Input error (constrained)\label{fig: constrained}]{%
        \includegraphics[width=0.5\linewidth]{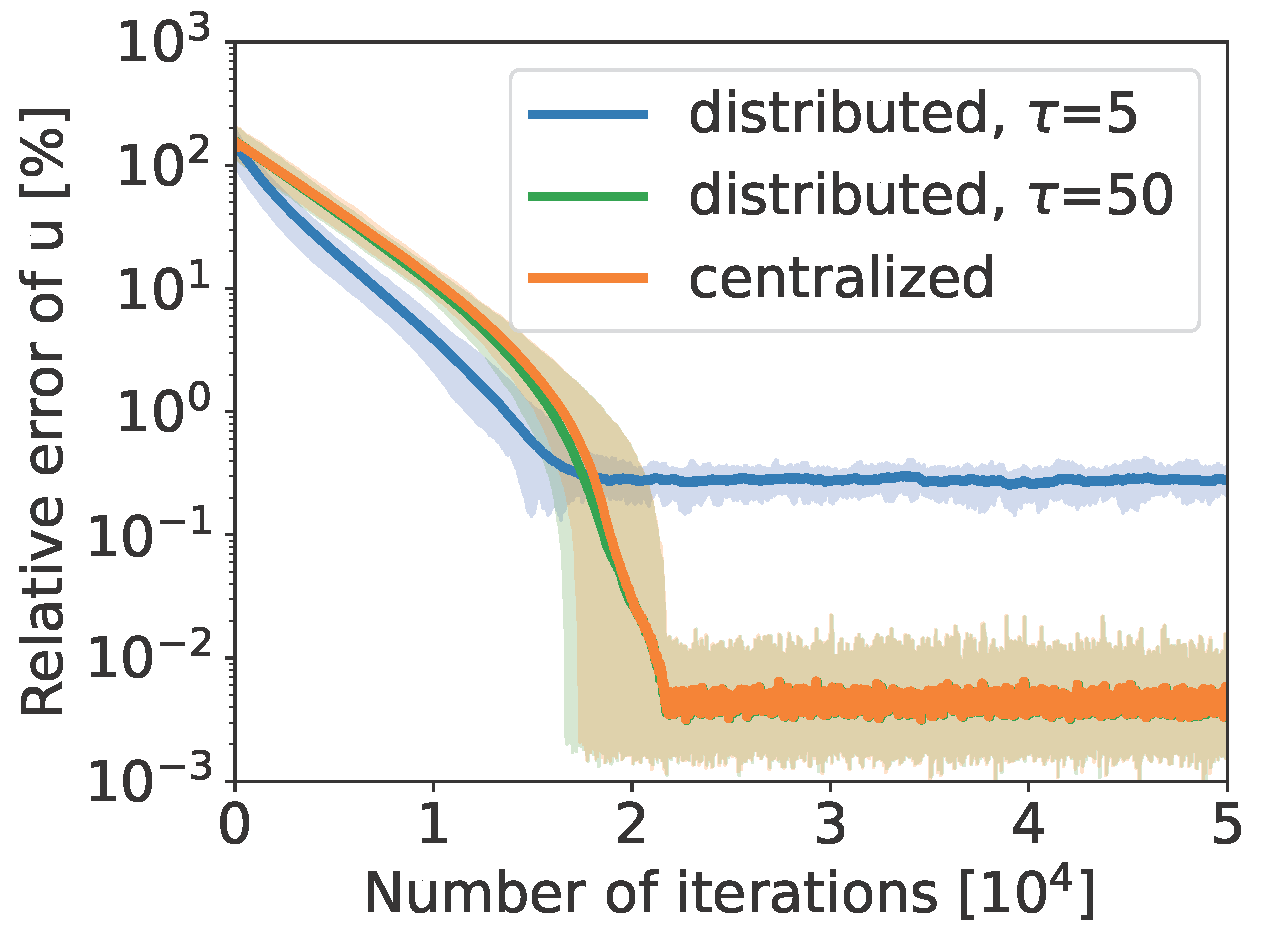}}\\
  \subfloat[Trajectory for constrained input $I_c(6)$ (averaged results)\label{fig: constrained_trajectory}]{%
        \includegraphics[width=1\linewidth]{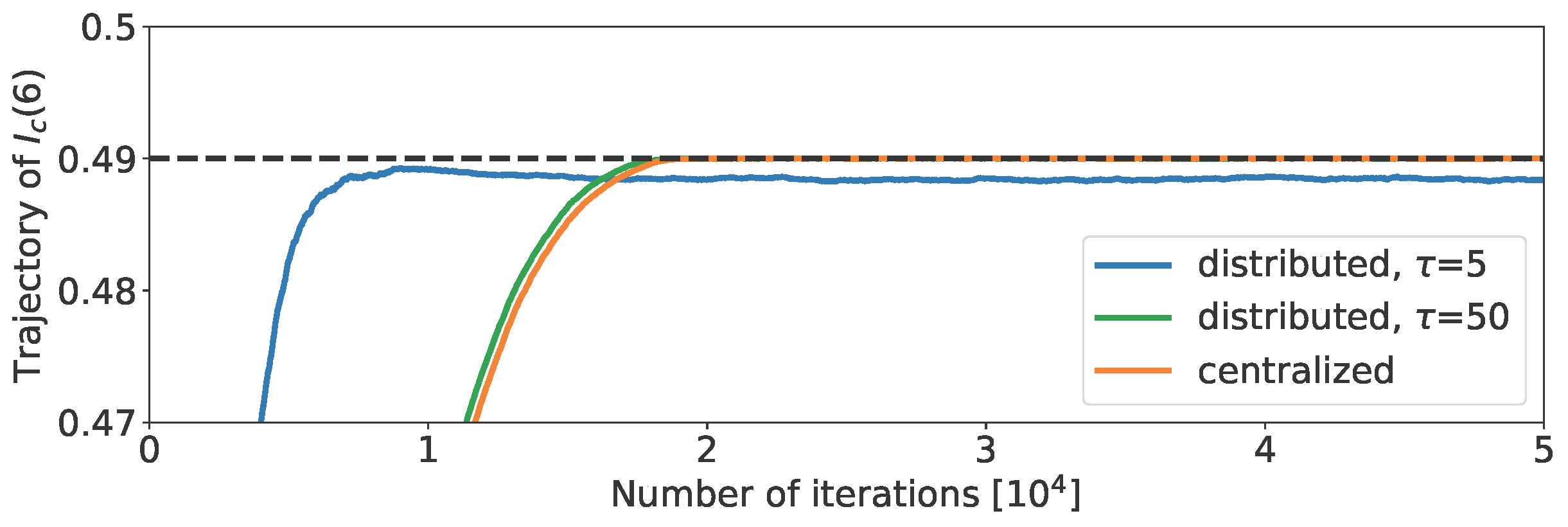}}
  \caption{Results for the distributed and centralized model-free controllers on a DC power system.}
  \label{fig: power system zeroth order}
\end{figure}

For the unconstrained problem, we plot the relative errors (i.e., $\norm{u_k-u^*}/\norm{u^*}$) in Fig. \ref{fig: unconstrained}.
The solid line represents the average value over $20$ independent numerical experiments. The blue curve and green curve correspond to the distributed model-free OFO controller \eqref{algo: delayed zeroth-order} with $\tau = 5$ and $\tau = 50$, respectively.
The orange curve represents the centralized counterpart (i.e., when $\tau$ is set to zero and $W^{\tau}$ is replaced by $\N$ in \eqref{algo: delayed zeroth-order}). They all own a linear convergence rate. 
The steady-state sub-optimality of the distributed controller \eqref{algo: delayed zeroth-order} results from the error of zeroth-order gradient estimates and the finite consensus iterations reflected in $W^{\tau}$. As $\tau$ increases from $5$ to $50$, the performance of the distributed controller approaches that of the centralized controller. 

For the scenario with input constraints, the convergence results are shown in Fig.~\ref{fig: constrained}. Similar to the unconstrained scenario, the distributed controller enjoys a fast convergence rate. The gap between the distributed and centralized controllers decreases as we increase $\tau$. Constraint satisfaction is illustrated by the averaged trajectory of the input $I_c(6)$ in Fig.~\ref{fig: constrained_trajectory}, where the black dashed line represents the upper bound on the input. By exchanging more data (i.e., increasing $\tau$), the distributed controller converges to the optimal point, which is at the boundary of the feasible set.



\section{Conclusion}

In this paper, we proposed a distributed online feedback optimization controller to optimize the steady-state performance of a networked nonlinear system. This controller features fully model-free operations, moderate costs of local computations and communication, and closed-loop guarantees on optimality and constraint satisfaction. These benefits were achieved via a consensus-based protocol that exploits a queue of local objective values to track the global objective, as well as iterative local updates involving gradient estimates and projection. Future directions include asynchronous implementations to handle delays or packet drops and extensions to other applications, e.g., traffic networks and process control systems.

\bibliographystyle{IEEEtran}
\bibliography{references}

\balance
\appendix

\subsection{Proof of a Supporting Lemma}
\label{appendix: proof of support lemma}
We first provide a useful lemma, which will be implemented in the following proofs. Let $\mathcal{F}_k$ be the $\sigma$-algebra (see \cite[Section 1.2]{ash2000probability}) on the set formed by the random signals $v_0, \ldots, v_{k-1}$. From \cite[Lemma 5]{Zhang2020}, we have
\begin{equation*}
    \Ef{\gest \obj(u_{k-\tau})}{k-\tau} = \nabla \obj_{\delta}(u_{k-\tau}).
\end{equation*}
We further define 
\begin{align*}
    \E{\norm{\gest \obj(u_{0})\! +\! e_{0}}^2}\!=\!\max_{p,q\in[\tau]} \{\E{\norm{\frac{1}{\delta} (W^{p}\varphi_0\!-\!W^{q}\varphi^{\text{init}}_0)\odot v_0}^2}\},
\end{align*}
where $\varphi^{\text{init}}_0 = [\obj_1(u_0+\delta v^{\text{init}}_0),\hdots,\obj_N(u_0+\delta v^{\text{init}}_0)]^{\top}$, $v^{\text{init}}_0$ and $v_0$ both follow $\mathcal{N}(0,I_{N\times N})$.

\begin{lemma}
\label{lemma: zeroth-order technical 1}
    Let $i,j \in \{1,\ldots,N\}$. 
    If $1<k< \tau + 2$,
    \begin{flalign*}
        \E{(\obj_j(&u_k+\delta v_k)-\obj_j(u_{k-1}+\delta v_{k-1}))^2v_{k}(i)^2}\\
        \leq&2L_0^2\eta^2\E{\|\gest\obj(u_0)+e_0\|^2} + 8L_0^2\delta^2(N+1).&
    \end{flalign*}
    If $k \geq \tau + 2$, 
    \begin{flalign*}
    \E{(\obj_j(&u_k+\delta v_k)-\obj_j(u_{k-1}+\delta v_{k-1}))^2v_{k}(i)^2}\\
    \leq&2L_0^2\eta^2\E{\|\gest\obj(u_{k-\tau-1})+e_{k-\tau-1}\|^2} + 8L_0^2\delta^2(N+1).&
    \end{flalign*}
\end{lemma}
\begin{proof}
    The upper bound is derived as follows
    \begin{flalign*}
        \E{(&\obj_j(u_k+\delta v_k)-\obj_j(u_{k-1}+\delta v_{k-1}))^2v_{k}(i)^2}\\
        = & \E{\Ef{(\obj_j(u_k+\delta v_k)-\obj_j(u_{k-1}+\delta v_k)\\
        &+\obj_j(u_{k-1}+\delta v_k)-\obj_j(u_{k-1}+\delta v_{k-1}))^2v_{k}(i)^2}{k}}\\
        \leq & 2L_0^2\E{\Ef{\|u_k-u_{k-1}\|^2v_{k}(i)^2}{k}}\\
        &+2L_0^2\delta^2\E{\Ef{\|v_k-v_{k-1}\|^2v_{k}(i)^2}{k}}\\
        \leq & 2L_0^2\E{\Ef{\|u_k-u_{k-1}\|^2v_{k}(i)^2}{k}}+8L_0^2\delta^2(N+1).&
    \end{flalign*}
    If $1<k< \tau + 2$,
    \begin{flalign*}
        \E{(&\obj_j(u_k+\delta v_k)-\obj_j(u_{k-1}+\delta v_{k-1}))^2v_{k}(i)^2}\\
        \leq& 2L_0^2\eta^2\E{\|\gest\obj(u_0)+e_0\|^2}+8L_0^2\delta^2(N+1).&
    \end{flalign*}
    If $k \geq \tau + 2$,
    \begin{flalign*}
        \E{(&\obj_j(u_k+\delta v_k)-\obj_j(u_{k-1}+\delta v_{k-1}))^2v_{k}(i)^2}\\
        \leq& 2L_0^2\eta^2\E{\|\gest\obj(u_{k\!-\!\tau\!-\!1})\!+\!e_{k\!-\!\tau\!-\!1}\|^2}\!+\!8L_0^2\delta^2(N\!+\!1).&\qedhere
    \end{flalign*} 
\end{proof}

\subsection{Proof of Lemma \ref{lemma 1}}
\label{appendix: proof of technical results}

\begin{proof}
We first provide an upper bound on $\E{\|e_k\|^2}$
\begin{flalign*}
    \E{\|&e_k\|^2}\\
    =&\frac{1}{\delta^2}\bb{E}\Big[\bb{E}\Big[\sum_{i=1}^{N}v_{k}(i)^2\Big(\sum_{j=1}^{N}(\Delta W)_{ij}(\varphi_{j,k}-\varphi_{j,k-1})\Big)^2|\mathcal{F}_k\Big]\Big]\\
    \leq &\frac{N}{\delta^2}\bb{E}\Big[\bb{E}\Big[\sum_{i=1}^{N}v_{k}(i)^2\Big(\sum_{j=1}^{N}(\Delta W)_{ij}^2(\varphi_{j,k}-\varphi_{j,k-1})^2\Big)|\mathcal{F}_k\Big]\Big],&
\end{flalign*}
where $\Delta W= W^{\tau}-\N$.
From Lemma \ref{lemma: zeroth-order technical 1}, if $1 < k < \tau + 2$,
\begin{flalign*}
    \E{\|e_k&\|^2}\\
    \leq&\frac{N}{\delta^2}\left(2L_0^2\eta^2\E{\|\gest\obj(u_0)+e_0\|^2}+8L_0^2\delta^2(N+1)\right)\\
    &\cdot\sum_{i=1}^{N}\sum_{j=1}^{N}(W^{\tau}-\N)_{ij}^2\\
    \leq&\frac{2N}{\delta^2}L_0^2\eta^2\text{tr}[(W^{\tau}-\N)^{2}]\E{\|\gest\obj(u_0)+e_0\|^2}\\
    + &8L_0^2\text{tr}[(W^{\tau}-\N)^{2}](N+4)^2.&
\end{flalign*}
When $k \geq \tau + 2$,
\begin{flalign*}
    \E{\|e_k&\|^2}\\
    \leq& \frac{2N}{\delta^2}L_0^2\eta^2\text{tr}[(W^{\tau}\!-\!\N)^{2}]\E{\|\gest\obj(u_{k\!-\!\tau\!-\!1})\!+\!e_{k\!-\!\tau\!-\!1}\|^2}\\
    + &8L_0^2\text{tr}[(W^{\tau}-\N)^{2}](N+4)^2.&
\end{flalign*}
For $\E{\norm{\gest\obj(u_k)}^2}$, by
following a similar reasoning as the proof for Lemma \ref{lemma: zeroth-order technical 1}, we have
\begin{equation*}
    \E{\norm{\gest\obj(u_k)}^2}\leq\frac{2N}{\delta^2}L_0^2\E{\Ef{\norm{u_k-u_{k-1}}^2}{k}}+8L_0^2(N+4)^2.
\end{equation*}
If $1 < k < \tau + 2$,
\begin{flalign*}
    \E{\norm{\gest\obj(u_k)}^2}\leq&\frac{2N}{\delta^2}L_0^2\eta^2\E{\|\gest\obj(u_0)+e_0\|^2}+8L_0^2(N+4)^2.
\end{flalign*}
If $k \geq \tau + 2$,
\begin{flalign*}
    \E{\norm{\gest\obj(u_k)}^2}\leq&\frac{2N}{\delta^2}L_0^2\eta^2\E{\|\gest\obj(u_{k-\tau -1})+e_{k-\tau -1}\|^2}\\
    &+8L_0^2(N+4)^2.&
\end{flalign*}
For $\E{\|\gest\obj(u_k)+e_k\|^2}$, when $1<k < \tau + 2$, we can establish an upper bound as follows
\begin{flalign*}
    \E{\|&\gest\obj(u_k)+e_k\|^2}\leq2\E{\|\gest\obj(u_k)\|^2] + 2\bb{E}[\|e_k\|^2}\\
    \leq&2(\frac{2N}{\delta^2}L_0^2\eta^2\E{\|\gest\obj(u_0)+e_0\|^2} + 8L_0^2(N+4)^2)\\
    &+2(\frac{2N}{\delta^2}L_0^2\eta^2\text{tr}[(W^{\tau}-\N)^{2}]\E{\|\gest\obj(u_0)+e_0\|^2]\\
    &+ 8L_0^2\text{tr}[(W^{\tau}-\N)^{2}}(N+4)^2)\\
    =&\frac{4N}{\delta^2}L_0^2\eta^2\text{tr}[W^{2\tau}]\bb{E}[\|\gest\obj(u_0)\!+\!e_0\|^2]\!+ \!16L_0^2\text{tr}[W^{2\tau}](N\!+\!4)^2\\
    \leq&\frac{1}{\alpha}\E{\|\gest\obj(u_0)+e_0\|^2}+ 16L_0^2\text{tr}[W^{2\tau}](N+4)^2\frac{1}{1-\alpha},&
\end{flalign*}
where $\alpha = \frac{4N}{\delta^2}L_0^2\eta^2\text{tr}[W^{2\tau}]$, and the parametric conditions ensure that $\alpha < 1$.When $k \geq \tau + 2$, we similarly have
\begin{flalign*}
    \E{\|\gest&\obj(u_k)+e_k\|^2}\\
    \leq&\frac{4N}{\delta^2}L_0^2\eta^2\text{tr}[W^{2\tau}]\cdot\E{\|\gest\obj(u_{k-\tau-1})+e_{k-\tau-1}\|^2}\\
    &+ 16L_0^2\text{tr}[W^{2\tau}](N+4)^2\\
    \overset{\text{(s.1)}}{\leq}&\frac{1}{\alpha}\E{\|\gest\obj(u_0)+e_0\|^2}+ 16L_0^2\text{tr}[W^{2\tau}](N+4)^2\frac{1}{1-\alpha}.&
\end{flalign*}
The right-hand side of (s.1) is $R$ in Lemma~\ref{lemma 1}.

    For $\E{\|\gest\obj(u_k)\|^2}$, we have
    \begin{flalign*}
        \E{\|\gest\obj(u_k)\|^2}\leq&\frac{2N}{\delta^2}L_0^2\eta^2R + 8L_0^2(N+4)^2\\
        \overset{\text{(s.1)}}{=} &\frac{\alpha R}{2\text{tr}[W^{2\tau}]} + 8L_0^2(N+4)^2.&
    \end{flalign*}
    The right-hand side of (s.1) is $R_f$ in Lemma~\ref{lemma 1}.
    
    For $\E{\|e_k\|^2}$, we have
    \begin{flalign*}
        \E{\|e_k\|^2}\leq&\frac{2N}{\delta^2}L_0^2\eta^2\text{tr}[(W^{\tau}-\N)^{2}]R\\
        &+8L_0^2\text{tr}[(W^{\tau}-\N)^{2}](N+4)^2\\
        \overset{\text{(s.1)}}{=}&\text{tr}[(W^{\tau}-\N)^{2}]R_f.&
    \end{flalign*}
    The right-hand side of (s.1) is $R_e$ in Lemma~\ref{lemma 1}. 
\end{proof}

\subsection{Proof for Theorem \ref{theorem: delayed inexact gradient estimate}}
\label{appendix: proof of theorem 1}

\begin{proof}[Proof]
For $k \geq \tau + 1$, we can bound $\E{\norm{u_{k+1}-u^*}^2}$ via
    \begin{align*}
        \E{\norm{&u_{k+1}-u^*}^2}\\
        =&\E{\norm{u_k - u^* - \eta(\gest \obj (u_{k-\tau}) + e_{k-\tau})}^2}\\
        =&\E{\|u_k - u^*\|^2 - 2\eta(\gest \obj (u_{k-\tau})+e_{k-\tau})^{\top}(u_k-u^*)\\
        &\quad + \eta^2\|\gest \obj (u_{k-\tau}) + e_{k-\tau}\|^2}\\
        \leq&\E{\|u_k \!-\! u^*\|^2 - 2\eta\gest \obj (u_{k-\tau})^{\top}(u_k-u^*)\\
        &\quad+ \eta \|e_{k-\tau}\|^2  + \eta\|u_k-u^*\|^2+ \eta^2\|\gest \obj (u_{k-\tau}) + e_{k-\tau}\|^2}.
    \end{align*}
    We use Lemma \ref{lemma 1} to obtain
    \begin{flalign*}
        \E{&\norm{u_{k+1}-u^*}^2}\leq\E{(1+\eta)\|u_k - u^*\|^2}\\
        &- \E{2\eta\gest \obj (u_{k-\tau})^{\top}(u_k-u^*)} + \eta R_e + \eta^2R\\
        =&\E{(1+\eta)\|u_k - u^*\|^2}+ \eta R_e + \eta^2R\\
        &- 2\eta\underbrace{\E{ \gest \obj (u_{k-\tau})^{\top}(u_{k-\tau}-u^*)}}_{\text{(a)}}\\
        &-2\eta\underbrace{\E{ \nabla \obj (u_k)^{\top}(u_k-u_{k-\tau})}}_{\text{(b)}}\\
        &-2\eta\underbrace{\E{(\gest \obj (u_k)-\nabla \obj (u_k))^{\top}(u_k-u_{k-\tau})}}_{\text{(c)}}\\
        &-2\eta\underbrace{\E{(\gest \obj(u_{k-\tau})-\gest \obj(u_k))^{\top}(u_k-u_{k-\tau})}}_{\text{(d)}}.&
    \end{flalign*}
    We establish upper bounds on terms (a)-(d). For (a), we have
    \begin{flalign*}
        -\E{\gest \obj (u_{k-\tau})^{\top}& (u_{k-\tau}-u^*)}\overset{\text{(s.1)}}{\leq}\E{\obj_{\delta}(u^*) - \obj_{\delta}(u_{k-\tau})}\\
        \overset{\text{(s.2)}}{\leq}&\E{\obj^* -\obj (u_{k-\tau})} + \frac{L_1}{2}\delta^2N\\
        \overset{\text{(s.3)}}{\leq}&-\frac{m}{2}\E{\|u_{k-\tau}-u^*\|^2}+ \frac{L_1}{2}\delta^2N,&
    \end{flalign*}
    where (s.1) follows from the convexity of $\obj_{\delta}$; (s.2) uses the closeness between $\obj_\delta$ and $\obj$; (s.3) holds because $\obj (u_{k-\tau}) \geq \obj^* + \frac{m}{2}\|u_{k-\tau}-u^*\|^2$. For (b), we have
    \begin{flalign*}
        -\nabla \obj (u_k)^{\top}(u_k\! - \!u_{k-\tau})\leq&\obj (u_{k-\tau})\! -\!\obj(u_k)\! - \!\frac{m}{2}\|u_{k-\tau}\!-\!u_k\|^2\\
        \overset{\text{(s.1)}}{\leq}& \tau L_0 \sqrt{R} \eta-\frac{m}{2}\|u_{k-\tau}-u_k\|^2,&
    \end{flalign*}
    where in (s.1) we rewrite $\obj(u_k) - \obj (u_{k-\tau})$ as $\obj(u_k) - \obj(u_{k-1})+\obj(u_{k-1})- \cdots + \obj (u_{k-\tau + 1}) - \obj (u_{k-\tau})$ and use the convexity of $\obj$. For (c), we have 
    \begin{flalign*}
        \E{(\gest \obj& (u_k) - \nabla \obj (u_k))^{\top}(u_k - u_{k-\tau})}\\
        \leq&\eta\E{\frac{1}{2}\|\nabla \obj_{\delta}(u_k) - \nabla \obj (u_k)\|^2\\
        &+\frac{1}{2} \|\gest \obj (u_{k-1})+e_{k-1} + \dots + \gest \obj (u_{k-\tau})+e_{k-\tau}\|^2}\\
        \leq&\frac{1}{2}\delta^2 NL_1^2 \eta + \frac{1}{2}\tau R \eta.&
    \end{flalign*}
    For (d), we have 
    \begin{flalign*}
        \E{(\gest \obj &(u_{k-\tau}) - \gest \obj (u_k))^{\top}(u_k - u_{k-\tau})}\\
        \leq&\eta\E{\norm{\gest \obj (u_{k-\tau}) - \gest \obj (u_k)}\\
        &\cdot\norm{\gest \obj (u_{k-1})+e_{k-1} + \dots + \gest \obj (u_{k-\tau})+e_{k-\tau}}}\\
        \leq&\eta\E{\norm{\gest \obj (u_{k-\tau})}^2+\norm{\gest \obj (u_k)}^2\\
        &+\frac{1}{2}\norm{\gest \obj (u_{k-1})+e_{k-1} + \dots + \gest \obj (u_{k-\tau})+e_{k-\tau}}^2}\\
        \leq&2\eta R_f + \frac{1}{2}\tau R\eta.&
    \end{flalign*}
    Hence, we can bound $\E{\norm{u_{k+1}-u^*}^2}$ by
    \begin{align}\label{eq:recursive_bd_unconstrained}
        \E{\norm{&u_{k+1}-u^*}^2}\overset{\text{(s.1)}}{\leq}\E{(1+\eta)\|u_k - u^*\|^2} \notag \\
        &- m\eta\E{(\|u_k - u^*\|^2+2(u_k-u_{k-\tau})^{\top}(u^*-u_{k-\tau}))} \notag \\
        &\!+ \!(2\tau L_0 \sqrt{R} \!+ \!\delta^2 L_1^2N \!+ \!2\tau R \!+ \!4R_f\! + \!R) \eta^2\!+ \!(R_e\!+\!L_1\delta^2N)\eta \notag \\
        \overset{\text{(s.2)}}{\leq}&(1-(m-1)\eta + m\eta^2)\E{\|u_k - u^*\|^2} + 2m\tau R\eta^3 \notag \\
        &+ (2\tau L_0 \sqrt{R} + \delta^2 L_1^2N + 2\tau R + 4R_f+ m\tau R  + R) \eta^2 \notag \\
        &+ (R_e+L_1\delta^2N)\eta,
    \end{align}
    where in (s.1) we incorporate the above upper bounds on (a)-(d). In (s.2), we bound $-\E{(u_k\!-\!u_{k\!-\!\tau})^{\top}(u^*\!-\!u_{k\!-\!\tau})}$ by
    \begin{flalign*}
        -\E{(&u_k-u_{k-\tau})^{\top}(u^*-u_{k-\tau})}\\
        \leq& \E{\|u_k\!-\!u_{k\!-\!1}\!+\!u_{k\!-\!1}-\dots \!+\!u_{k-\tau \!+\!1}\!-\!u_{k\!-\!\tau}\|\|u^*-u_{k-\tau}\|}\\
        =&\eta\E{\|\gest \obj (u_{k-1})+e_{k-1} + \dots + \gest \obj (u_{k-\tau})+e_{k-\tau}\|\\
        &\cdot\|u^*-u_k+u_k-u_{k-\tau}\|}\\
        \leq&\eta\E{(\frac{1}{2}+\eta)\|\gest \obj (u_{k-1})+e_{k-1}\\
        &+\dots + \gest \obj (u_{k-\tau})+e_{k-\tau}\|^2+\frac{1}{2}\|u^*-u_k\|^2}\\
        =&\frac{\eta}{2}\E{\|u^*-u_k\|^2} +(\frac{\eta}{2}+\eta^2)\tau R.&
    \end{flalign*}
We define $\rho=1-(m-1)\eta + m\eta^2$ and $p(\eta) = a_1\eta^3+a_2\eta^2+a_3\eta$, where
\begin{equation}
\label{bound on as}
    \begin{cases}
        a_1\! = \!2m\tau R,\\
        a_2\! = \!2\tau L_0 \sqrt{R} \!+\! \delta^2 L_1^2N \!+\! 2\tau R \!+\! 4R_f\!+\! m\tau R  \!+\! R,\\
        a_3\! = \!R_e+L_1\delta^2N.
    \end{cases}
\end{equation}
We telescope the inequality \eqref{eq:recursive_bd_unconstrained} and arrive at \eqref{eq:upper_bound_unconstrained}.
\end{proof}

\subsection{Proof for Corollary \ref{corollary one}}
\label{appendix: proof of corollary}
\begin{proof}
    From Lemma \ref{lemma 1}, we have 
    \begin{flalign*}
        \frac{p(\eta)}{1-\rho}=\frac{a_1\eta^2+a_2\eta+a_3}{(m-1)- m\eta}=\frac{a_1\eta^2+a_2\eta+R_e + L_1\delta^2N}{(m-1)- m\eta}.
    \end{flalign*}
    By invoking the parametric conditions of $\eta_\epsilon$ and $\delta$, we obtain
    \begin{flalign*}
        \frac{p(\eta_{\epsilon})}{1-\rho}=\frac{\epsilon}{2}+\frac{R_e}{(m-1)- m\eta}.
    \end{flalign*}
    If we select $\tau$ such that $\tau > \ln(\frac{(m-1-m\eta_{\epsilon})\epsilon}{2R_f(N-1)})/\ln(\lambda_2^2)$, then
    \begin{equation*}
        \frac{p(\eta_{\epsilon})}{1-\rho}<\frac{\epsilon}{2}+\frac{\epsilon}{2}<\epsilon. \qedhere
    \end{equation*}
\end{proof}

\subsection{Proof for Theorem \ref{theorem: projected dynamics}}
\label{appendix: proof of theorem 2}

\begin{proof}
We derive the following recursive inequality of the expected distance to the optimal point
\begin{flalign*}
\E{\norm{&u_{k+1}-u^*}}=\E{\norm{\text{Proj}_{\mathcal{U}}[u_k-\eta(\gest \obj(u_{k-\tau})+e_{k-\tau})]\\
&-\text{Proj}_{\mathcal{U}}[u^*-\eta(\gest \obj(u^*)+e^*)]}}\\
\stackrel{\text{(s.1)}}{\leq}&\E{\norm{u_k\!-\!u^*\!-\!\eta\nabla \obj(u_k)\!-\!\eta\nabla\obj(u^*)\!+\!\eta(\nabla \obj(u_k)\!+\!\nabla\obj(u^*)\\
&-(\gest \obj(u_{k-\tau})+e_{k-\tau})+(\gest \obj(u^*)+e^*))}}\\
\overset{\text{(s.2)}}{\leq}&\sqrt{1-2m\eta+L_1^2\eta^2}\E{\norm{u_k\!-\!u^*}}\!+\!\eta\E{\norm{\nabla \obj(u_k)\!+\!\nabla\obj(u^*)\\
&-(\gest \obj(u_{k-\tau})+e_{k-\tau})+(\gest \obj(u^*)+e^*}}\\
\overset{\text{(s.3)}}{\leq}&\sqrt{1-2m\eta+L_1^2\eta^2}\E{\norm{u_k-u^*}}+2\eta L_0\\
&+ 2\eta \sqrt{\frac{4ND_{\mathcal{U}}}{\delta^2}L_0^2\text{tr}[W^{2\tau}]+ 16L_0^2\text{tr}[W^{2\tau}](N+4)^2},&
\end{flalign*}
where (s.1) uses the non-expansiveness of the projection operator when the constraint set is convex; (s.2) follows by analyzing $\norm{u_k-u^*}^2$ using the properties of strongly convex and smooth objectives; (s.3) is because a differentiable Lipschitz function has a bounded gradient. Moreover, the upper bound on $\E{\norm{\gest \obj(u_{k-\tau})+e_{k-\tau}}}$ used in (s.3) is obtained by following similar reasoning as the proof for Lemma \ref{lemma 1} and noting that the diameter of the feasible set is $D_{\mathcal{U}}$. 
We recursively apply the above inequality. Hence, \eqref{eq:dist_opt_constrained} holds.
\end{proof}

\end{document}